\pdfoutput=1

\documentclass[11pt,reqno]{amsart}

\RequirePackage{pict2e}
\usepackage{amssymb,amscd,url,amsmath,amscd,enumerate}
\usepackage{mathrsfs}
\usepackage{xcolor}    
\usepackage{bbm}       
\usepackage{extarrows}
\usepackage{constants} 
\usepackage{stackrel}
\usepackage{tikz-cd}
\usetikzlibrary{arrows.meta}
\usepackage{centernot}
\usepackage{upref}     
\usepackage{mathtools}
\usepackage[T1]{fontenc}
\usepackage{graphicx}
\usepackage{amsthm}
\RequirePackage[pdfa, psdextra]{hyperref}  
\hypersetup{
    colorlinks=true,
    linkcolor=blue,
    urlcolor=cyan,
    }

\usepackage{color} 
\definecolor{darkgreen}{rgb}{0.0, 0.7, 0.0}

\begin{document}

\allowdisplaybreaks


\title[\fontsize{8}{12}\selectfont Generalized Franchetta conjecture in genus 11]
{On O'Grady's generalized Franchetta conjecture for genus 11 K3 surfaces}
\date{\today}

\author[Yuan Lu]{Yuan Lu}
\address{School of Mathematical Sciences, Peking University}
\email{luyuan@stu.pku.edu.cn}


\newtheorem{theorem}{Theorem}[section]
\newtheorem{lemma}[theorem]{Lemma}
\newtheorem{sublemma}[theorem]{Sublemma}
\newtheorem{conjecture}[theorem]{Conjecture}
\newtheorem{proposition}[theorem]{Proposition}
\newtheorem{corollary}[theorem]{Corollary}
\newtheorem{property}[theorem]{Property}
\newtheorem{hypothesis}[theorem]{Hypothesis}
\theoremstyle{definition}
\newtheorem*{claim}{Claim}
\newtheorem{definition}[theorem]{Definition}
\newtheorem*{intuition}{Intuition}
\newtheorem{example}[theorem]{Example}
\newtheorem{remark}[theorem]{Remark}
\newtheorem{question}[theorem]{Question}
\newtheorem{cheat}[theorem]{Cheat}
\theoremstyle{remark}
\newtheorem*{acknowledgement}{Acknowledgements}
\newtheorem{exercise}[theorem]{exercise}



\newenvironment{notation}[0]{%
  \begin{list}%
    {}%
    {\setlength{\itemindent}{0pt}
     \setlength{\labelwidth}{4\parindent}
     \setlength{\labelsep}{\parindent}
     \setlength{\leftmargin}{5\parindent}
     \setlength{\itemsep}{0pt}
     }%
   }%
  {\end{list}}

\newenvironment{parts}[0]{%
  \begin{list}{}%
    {\setlength{\itemindent}{0pt}
     \setlength{\labelwidth}{1.5\parindent}
     \setlength{\labelsep}{.5\parindent}
     \setlength{\leftmargin}{2\parindent}
     \setlength{\itemsep}{0pt}
     }%
   }%
  {\end{list}}
\newcommand{\Part}[1]{\item[\upshape#1]}
\newcommand{\BB}{\mathbb{B}}
\newcommand{\CC}{\mathbb{C}}
\newcommand{\DD}{\mathbb{D}}
\newcommand{\EE}{\mathbb{E}}
\newcommand{\FF}{\mathbb{F}}
\newcommand{\GG}{\mathbb{G}}
\newcommand{\HH}{\mathbb{H}}
\newcommand{\II}{\mathbb{I}}
\newcommand{\JJ}{\mathbb{J}}
\newcommand{\KK}{\mathbb{K}}
\newcommand{\LL}{\mathbb{L}}
\newcommand{\MM}{\mathbb{M}}
\newcommand{\NN}{\mathbb{N}}
\newcommand{\OO}{\mathbb{O}}
\newcommand{\PP}{\mathbb{P}}
\newcommand{\QQ}{\mathbb{Q}}
\newcommand{\RR}{\mathbb{R}}
\newcommand{\TT}{\mathbb{T}}
\newcommand{\UU}{\mathbb{U}}
\newcommand{\VV}{\mathbb{V}}
\newcommand{\WW}{\mathbb{W}}
\newcommand{\XX}{\mathbb{X}}
\newcommand{\YY}{\mathbb{Y}}
\newcommand{\ZZ}{\mathbb{Z}}
\newcommand{\cA}{\mathcal A}
\newcommand{\cB}{\mathcal B}
\newcommand{\cC}{\mathcal C}
\newcommand{\cD}{\mathcal D}
\newcommand{\cE}{\mathcal E}
\newcommand{\cF}{\mathcal F}
\newcommand{\cG}{\mathcal G}
\newcommand{\cH}{\mathcal H}
\newcommand{\cI}{\mathcal I}
\newcommand{\cJ}{\mathcal J}
\newcommand{\cK}{\mathcal K}
\newcommand{\cL}{\mathcal L}
\newcommand{\cM}{\mathcal M}
\newcommand{\cN}{\mathcal N}
\newcommand{\cO}{\mathcal O}
\newcommand{\cP}{\mathcal P}
\newcommand{\cQ}{\mathcal Q}
\newcommand{\cR}{\mathcal R}
\newcommand{\cS}{\mathcal S}
\newcommand{\cT}{\mathcal T}
\newcommand{\cU}{\mathcal U}
\newcommand{\cV}{\mathcal V}
\newcommand{\cW}{\mathcal W}
\newcommand{\cX}{\mathcal X}
\newcommand{\cY}{\mathcal Y}
\newcommand{\cZ}{\mathcal Z}

\newcommand{\al}{\alpha}
\newcommand{\be}{\beta}
\newcommand{\ga}{\gamma}
\newcommand{\Ga}{\Gamma}
\newcommand{\de}{\delta}
\newcommand{\De}{\Delta}
\newcommand{\ep}{\epsilon}

\newcommand{\ze}{\zeta}
\newcommand{\te}{\theta}
\newcommand{\Te}{\Theta}
\newcommand{\io}{\iota}
\newcommand{\ka}{\kappa}
\newcommand{\la}{\lambda}
\newcommand{\La}{\Lambda}
\newcommand{\si}{\sigma}
\newcommand{\vp}{\varphi}
\newcommand{\om}{\omega}
\newcommand{\Om}{\Omega}
\newcommand{\mr}{\mathrm}
\newcommand{\pr}{\mathrm{pr}}
\newcommand{\red}{\mathrm{red}}
\newcommand{\alg}{\mathrm{alg}}
\newcommand{\tor}{\mathrm{tor}}
\newcommand{\rat}{\mathrm{rat}}
\newcommand{\id}{\mathrm{id}}
\newcommand{\M}{\mathrm{M}}

\newcommand{\Hom}{\operatorname{Hom}}
\newcommand{\Pic}{\operatorname{Pic}}
\newcommand{\codim}{\operatorname{codim}}
\newcommand{\Ext}{\operatorname{Ext}}
\newcommand{\NS}{\operatorname{NS}} 
\newcommand{\im}{\operatorname {Im}}
\newcommand{\CH}{\operatorname {CH}}
\newcommand{\Jac}{\operatorname {Jac}}
\newcommand{\Gr}{\operatorname{Gr}}
\newcommand{\GL}{\operatorname{GL}}
\newcommand{\GSp}{\operatorname{GSp}}
\newcommand{\Spec}{\operatorname{Spec}}
\newcommand{\End}{\operatorname{End}}
\newcommand{\Aut}{\operatorname{Aut}}
\newcommand{\Proj}{\operatorname{Proj}}

\newcommand{\ba}{\backslash}
\newcommand{\su}{\subset}
\newcommand{\pa}{\partial}
\newcommand{\ov}{\overline}
\newcommand{\ti}{\tilde}
\newcommand{\wt}{\widetilde}
\newcommand{\tm}{\times}
\newcommand{\ot}{\otimes}
\newcommand{\op}{\oplus}
\newcommand{\bt}{\bigotimes}
\newcommand{\bp}{\bigoplus}
\newcommand{\ci}{\circ}

\begin{abstract}
    O'Grady's generalized Franchetta conjecture asks whether any codimension two cycle on the universal polarized K3 surface restricts to a multiple of the Beauville–Voisin class on a given K3 surface. We apply Mukai's program for genus $11$ curves and K3 surfaces, together with a result on the tautological generation of the second Chow group of the moduli space of curves, to give an affirmative answer to this conjecture in genus $11$.
\end{abstract}

\maketitle

\tableofcontents
\section{Introduction}
Throughout, we work over the field of complex numbers. Recall that for a projective K3 surface $S$, Beauville and Voisin showed in \cite{BV} that there exists a canonical class $o_S$ in the Chow group $\CH^2(S)$ that satisfies the following properties:
\begin{enumerate}
    \item For each closed point $x\in S$ on a (possibly singular) rational curve $C\subset S$, we have $[x]=o_S\in \CH^2(S)$.
    \item For any divisors $D_1,D_2\in \CH^1(S)$, the intersection $D_1\cdot D_2$ is a multiple of $o_S$.
    \item The second Chern class $c_2(S)=24o_S$.
\end{enumerate}
The class $o_S$ is called the Beauville--Voisin class of the K3 surface $S$.\\
\indent For each $g\ge 3$, denote $\cF_g$ the moduli space of primitively polarized K3 surface of degree $2g-2$ (or equivalently, genus $g$). Then we can take $\cF_g'\su\cF_g$ as the smooth dense open subset parametrizing primitively polarized K3 surfaces with trivial automorphism groups. Hence, we have a universal K3 surface $\pi: \cS_g'\to \cF_g'$.\\
\indent Motivated by the Franchetta conjecture on the moduli space of curves \cite{AC87}, O'Grady presented the generalized Franchetta conjecture on the moduli space of K3 surfaces in \cite[Section 5]{OG}, which is stated as follows.
\begin{conjecture}[Generalized Franchetta Conjecture]
\label{Generalized Franchetta Conjecture}
     Assume $g\ge3$. For each closed point $x\in \cF_g'$, denote $S_x:=\pi^{-1}(x)$, then for each class $a\in \CH^2(\cS_g')$, the restriction~$a|_{S_x}$ is a multiple of the Beauville--Voisin class $o_{S_x}$.
\end{conjecture}
At present, Conjecture \ref{Generalized Franchetta Conjecture} is largely open. So far, it is known in the following cases: in \cite{PSY}, Pavic--Shen--Yin proved Conjecture \ref{Generalized Franchetta Conjecture} for $3\le g\le 10$ and $g=12,13,16,18,20$, which are the genera $g\ge3$ with known Mukai models; in \cite{FL}, Fu--Laterveer used  special cubic fourfolds to prove Conjecture \ref{Generalized Franchetta Conjecture} for $g=14$. These results leave a clear gap: genus~$g=11$ is the lowest genus for which neither a Mukai model is known, nor does it have a nice correspondence with special cubic fourfolds.\\
\indent The main result of this paper is to resolve this case by introducing a third, completely different strategy.
\begin{theorem}
\label{My work}
    Conjecture \ref{Generalized Franchetta Conjecture} is true for genus $g=11$.
\end{theorem}
Our method is based on exploiting Mukai's program for genus $11$ curves and K3 surfaces, which provides a birational bridge between the moduli space of K3 surfaces and the moduli space of curves. The main result of Mukai's program in genus $11$ is stated as follows:
\begin{theorem}[Mukai's program in genus $11$ \cite{Muk}]
\label{Mukai's program}
    Let $\cP$ be the $\PP^{11}$-bundle over $\cF_{11}'$ parametrizing 
$$\{[S,\mathcal{L},C]\mid(S,\mathcal{L})\in 
\mathcal{F}_{11}',C \in |\mathcal{L}|\}.$$
Then the rational map $\phi:\cP\dashrightarrow \cM_{11}$ mapping $[S,\cL,C]$ to the curve $C$ is a birational map. Here, $\cM_{11}$ is the moduli stack of smooth curves of genus $11$.
\end{theorem}
\indent This birational equivalence allows us to transfer the problem to a problem about the moduli space of curves. Specifically, we show that the Chow group $\CH^2(\cS_{11}')_\QQ$ can be controlled by the Chow group $\CH^2(\cF_{11}')_\QQ$ and the Chow group $\CH^2(\ov{M}_{11,1})_\QQ$ of the coarse moduli space $\ov{M}_{11,1}$ of stable genus $11$ curves with one marked point.\\
\indent Another key ingredient is the following result on the tautological generation of the second Chow group of the coarse moduli space of pointed curves.
\begin{theorem}
\label{CH2 is tautologically generated for rationally connected moduli space introduction}
    Suppose $2g-2+n>0$, and the coarse moduli space $\ov{M}_{g,n}$ is rationally connected. Then $\CH^2(\ov{M}_{g,n})_\mathbb{Q}$ is generated by tautological classes.
\end{theorem}
By \cite[Corollary 2]{Bar}, the coarse moduli space $\ov{M}_{11,1}$ is rationally connected. Therefore,~$\CH^2(\ov{M}_{11,1})_\QQ$ is tautologically generated. This complete description of $\CH^2(\ov{M}_{11,1})_\QQ$, combined with the bridge provided by Mukai's program, allows us to conclude the proof of Theorem \ref{My work}.\\
\indent This article is organized as follows. In Section \ref{Section: Reduction to moduli of curve}, we exploit Mukai's program to relate the generalized Franchetta conjecture in genus $11$ to the coarse moduli space $\ov{M}_{11,1}$, with the main result stated in Theorem \ref{P11 is open subset of coarse noduli M11 after deleting codim 2 subsets}. In Section \ref{Section: result on chow gp}, we focus on the proof of Theorem \ref{CH2 is tautologically generated for rationally connected moduli space introduction}. In Section \ref{Section: Last section}, we combine these results to conclude the proof of Theorem \ref{My work}.\\
\\
\indent \textbf{Acknowledgments} The author is deeply grateful to his advisor, Qizheng Yin, for suggesting the topic of this project and for his invaluable guidance and insight throughout its completion. His encouragement was essential to this work.\\
\indent The author also thanks Samir Canning, Chunyi Li, Zhiyuan Li, Zhiyu Tian and Liang Xiao for useful conversations. The author is especially grateful to Yuchen Liu for a very helpful discussion, which led directly to the proof of Theorem \ref{CH2 is tautologically generated for rationally connected moduli space introduction}.\\
\indent This work was partially carried out during the author's visit to the University of Chicago during the REU program. The author thanks Peter May for his hospitality and for organizing the program.\\
\indent The author is also indebted to his parents and family for their unwavering support throughout his studies.
\section{Connection with the coarse moduli space of genus 11 curves}
\label{Section: Reduction to moduli of curve}
\indent In this section, we exploit Mukai's program to establish the bridge between the generalized Franchetta conjecture in genus $11$ and the coarse moduli space $\ov{M}_{11,1}$. The main result of this section is Theorem \ref{P11 is open subset of coarse noduli M11 after deleting codim 2 subsets}.\\ 
\indent We begin by introducing the geometric setup. Recall that $\cF_{11}'$ is a smooth dense open subset of the moduli space of primitively polarized K3 surface of genus $11$, parametrizing the ones with trivial automorphism groups, and $\pi:\cS_{11}'\to \cF_{11}'$ is the universal K3 surface.\\
\indent Let $\cP$ be the $\PP^{11}$-bundle over $\cF_{11}'$ parametrizing 
$$\{(S,\mathcal{L},C)\mid(S,\mathcal{L})\in 
\mathcal{F}_{11}',C \in |\mathcal{L}|\}.$$
Define the variety $\cC$ parametrizing 
$$\{(S,\mathcal{L},C,x)\mid(S,\mathcal{L})\in 
\mathcal{F}_{11}',C \in |\mathcal{L}|,x \in C\}.$$
We have natural projections $\wt{p}:\cC\to \cS_{11}'$ and $\wt{\pi}:\cC\to \cP$. The morphism $\wt{p}$ is a $\PP^{10}$-bundle, and the morphism $\wt{\pi}$ is a family of curves of arithmetic genus $11$.\\
\indent By Bertini's theorem, general fibers of $\wt{\pi}$ are smooth. Hence $\wt{\pi}$ induces a rational map $\phi:\cP\dashrightarrow \cM_{11}$, where $\cM_{11}$ is the moduli stack of smooth curves of genus $11$. Mukai's program in genus $11$ \cite{Muk} shows that $\phi$ is birational. The rational map $\phi$ further induces a rational map $\cC\dashrightarrow \cC_{11}$, where $\cC_{11}$ is the universal curve over $\cM_{11}$.\\
\indent Combining these constructions together, we get the following commutative diagram
\begin{equation}
\label{commutative diagram}
\begin{tikzcd}
 & & \mathcal{C}_{11} \arrow[d]\\
 \cC \arrow[r,"\wt{\pi}"] \arrow[rru,dashed, bend left] \arrow[d,"\wt{p}"]& \mathcal{P} \arrow[r,dashed,"\phi"] \arrow[d,"p"]& \mathcal{M}_{11} \\
 \mathcal{S}_{11}' \arrow[r,"\pi"]& \mathcal{F}_{11}'. &
\end{tikzcd}
\end{equation}
\begin{proposition}
\label{refinement of moduli of genus 11 K3}
    There exists a dense open subset $\mathcal{F}_{11}^\circ \subset \mathcal{F}_{11}'$ such that if we denote $\cP^\circ:=p^{-1}(\mathcal{F}_{11}^\circ)$ and
    $\cC^\circ:=\wt{\pi}^{-1}(\cP^\circ)$, the morphism $\wt{\pi}:\cC^\circ \to \cP^\circ$ has geometrically integral fibers.
\end{proposition}
\begin{proof}
    By \cite[tag 055B]{Stack} and \cite[tag 0579]{Stack}, let 
    $W$ be the set of points $x\in \cP$ where~$\wt{\pi}^{-1}(x)$ is geometrically integral, then $W$ is a constructible subset of $\cP$. Let~$Z:=\mathcal{P} \backslash W$, then by Chevalley’s theorem, $p(Z)$ is a constructible subset of $\cF_{11}'$.\\
    \indent By \cite[Corollary 6.4.3]{Huy}, $\mathcal{F}_{11}'$ can be identified with a dense open subset of $\wt O(\Lambda_{10})\backslash D_{10}$ via the period map $\vp_{10}$. Here, $\Lambda_{10}$ is the lattice $E_8(-1)^{\oplus2}\oplus U^{\oplus2} \oplus \mathbb{Z}(-20)$, the space~$D_{10}$ is the period domain of $\PP(\La\otimes_{\ZZ}\CC)$, and $\wt O(\Lambda_{10})$ is the subgroup of the orthogonal group~$O(\Lambda_{10})$ that fixes the discriminant. The constructions of them can be found in \cite[Section 6]{Huy}.\\
    \indent Denote $\pr:D_{10} \to \wt O(\Lambda_{10})\backslash D_{10}$ the natural morphism. For any closed point $x \in \mathcal{F}_{11}'$, denote $S_x:=\pi^{-1}(x)$ and denote $\cL_x$ the polarization on $S_x$. Take $p\in D_{10}$ such that $\vp_{10}(x)=\mr{pr}(p)$, then the orthogonal complement of $\cL_x$ in $\Pic(S_x)$ is isomorphic to~$\Lambda_{10}\cap p^\perp$ as lattices.\\
    \indent Suppose that $\vp_{10}(x) \notin \cup_{v \in \Lambda_{10}}\pr(v^\perp \cap D_{10})$, then $\Pic(S_x)=\ZZ\cdot \cL_x$. Hence, for any curve $C \in |\cL_x|$, $C$ is integral. So we have $\vp_{10}(p(Z))\su \cup_{v\in \La_{10}} \pr(v^\perp \cap D_{10})$.\\
    \indent For any $v\in \La_{10}$, the set $\pr(v^\perp \cap D_{10})$ is an analytically closed subset in $\wt O(\Lambda_{10})\backslash D_{10}$ of (real) codimension $2$. Since $p(Z)$ is constructible, and $p(Z)\su \cup_{v\in \La_{10}} \vp_{10}^{-1}(\pr(v^\perp \cap D_{10}))$, we have $p(Z)$ is contained in a proper Zariski closed subset of $\mathcal{F}_{11}'$. Therefore, we can take a dense open subset $\mathcal{F}_{11}^\circ \subset \mathcal{F}_{11}'$ that does not meet $p(Z)$. The open subset $\cF_{11}^\ci$ satisfies the required property.
\end{proof}
\indent We now establish the following lemma, which is used in Theorem \ref{main thm 1}.
\begin{lemma}
\label{dimH^0(C,j*T_P10)=120}
Suppose $(S,L)\in \cF_{11}^\circ$, and let $j:S\hookrightarrow \mathbb{P}^{11}$ be the embedding induced by a set of generators of $H^0(S,\mathcal{L})$. Denote the coordinates of $\PP^{11}$ as $[x_0:\cdot\cdot \cdot :x_{11}]$, and denote $\mathbb{P}^{10}:=V(x_0)$.\\
\indent Consider the Cartesian diagram
    \begin{center}
    \begin{tikzcd}
        C \arrow[r,hook,"i"]\ar[d,hook,"j'"]\ar[dr, phantom, "\square"] & S \arrow[d,hook,"j"] \\
        \mathbb{P}^{10} \arrow[r,hook,"i_0"] & \mathbb{P}^{11}.
    \end{tikzcd}
     \end{center}
Then the restriction $j'^*:H^0(\mathbb{P}^{10},T_{\mathbb{P}^{10}}) \to H^0(C,j'^*T_{\mathbb{P}^{10}})$ is an isomorphism.
\end{lemma}
\begin{proof}
    We have a commutative diagram
    \begin{center}
    \begin{tikzcd}
         0 \arrow[r] & H^0(\mathbb{P}^{10},\mathcal{O}_{\mathbb{P}^{10}}) \arrow[r]\arrow[d]  & H^0(\mathbb{P}^{10},\mathcal{O}_{\mathbb{P}^{10}}(1)^{\oplus11}) \arrow[r]\arrow[d] & H^0(\mathbb{P}^{10},T_{\mathbb{P}^{10}} \arrow[r]\arrow[d]) \arrow[r]\arrow[d] &  0\\
         0 \arrow[r] & H^0(C,\mathcal{O}_{C}) \arrow[r]  & H^0(C,\mathcal{O}_{C}(1)^{\oplus11}) \arrow[r] & H^0(C,j'^*T_{\mathbb{P}^{10}}) & 
    \end{tikzcd}
    \end{center}
    where the 2 rows are exact.\\
    \indent By Proposition \ref{refinement of moduli of genus 11 K3}, $C$ is geometrically integral. So the left vertical map is an isomorphism. The middle map is also an isomorphism due to the following commutative diagram of short exact sequences.
    \begin{center}
    \begin{tikzcd}
        0 \arrow[r] & H^0(\PP^{11},\mathcal{O}_{\mathbb{P}^{11}}) \arrow[r]\arrow[d]  & H^0(\PP^{11},\mathcal{O}_{\mathbb{P}^{11}}(1)) \arrow[r]\arrow[d] & H^0(\PP^{10},\mathcal{O}_{\mathbb{P}^{10}}(1))\arrow[r]\arrow[d] & 0 \\
        0 \arrow[r] & H^0(S,\mathcal{O}_{S}) \arrow[r]  & H^0(S,\mathcal{O}_{S}(1)) \arrow[r] & H^0(C,\mathcal{O}_C(1)) \arrow[r] & 0.
    \end{tikzcd}
    \end{center}
    \indent Therefore, we only need to prove that the bottom row in the first diagram is exact on the right, or equivalently, the morphism $f:H^1(C,\mathcal{O}_C) \to H^1(C,\mathcal{O}_C(1)^{\oplus11})$ is injective.\\
    \indent Consider the morphism between short exact sequences
    \begin{center}
        \begin{tikzcd}
            0 \ar[r] & \cO_S(-1) \ar[r]\ar[d] & \cO_S \ar[r]\ar[d,"\vp"] & i_*\cO_C\ar[r]\ar[d,"\psi"] & 0\\
            0 \ar[r] & \cO_S^{\oplus11}\ar[r] & \cO_S(1)^{\oplus11} \ar[r] & i_*\cO_C(1)^{\oplus11}\ar[r] & 0
        \end{tikzcd}
    \end{center}
    where $\vp$ is given by $a\mapsto (ax_1,...,ax_{11})$, and the other vertical maps are induced by $\vp$. Then the morphism $f$ is induced by $\psi$.\\
    \indent By \cite[Corollary 8.2]{SD}, $H^1(S, \mathcal{O}_{S}(1))=0$. Using $H^1(S,\cO_S)=0$ and $$H^2(S,\cO_S(1)) \cong H^0(S,\cO_S(-1))^\vee=0,$$
    we have the following commutative diagram
    \begin{center}
    \begin{tikzcd}
    0 \arrow[r] & H^1(C,\mathcal{O}_C) \arrow[r]\arrow[d,"f"] & H^2(S,\mathcal{O}_S(-1)) \arrow[r,"h"]\arrow[d,"g"] & H^2(S, \mathcal{O}_S) \arrow[r]\ar[d] & 0\\
    0 \arrow[r] & H^1(C,\mathcal{O}_C(1)^{\oplus11}) \arrow[r]& H^2(S,\mathcal{O}_S^{\oplus11}) \arrow[r] & 0 &
    \end{tikzcd}
    \end{center}
    where the 2 rows are exact.\\
    \indent Therefore, we only need to prove $h|_{\ker g}$ is injective. By Serre duality, it is equivalent to show that
    $$H^0(S, \mathcal{O}_S) \to \mathrm{coker}(H^0(S,\mathcal{O}_S^{\oplus11}) \to H^0(S,\mathcal{O}_S(1)))$$ is surjective, where $\mathcal{O}_S^{\oplus11} \to \mathcal{O}_S(1)$ is given by $(a_1,...,a_{11})\mapsto a_1x_1+...+a_{11}x_{11}$, and~$\mathcal{O}_S \mapsto  \mathcal{O}_S(1)$ is given by $a \mapsto ax_0$. It is easy to check the surjectivity. This completes the proof of Lemma \ref{dimH^0(C,j*T_P10)=120}.
\end{proof}
For each closed point $(S,\cL,C)\in \cP^\circ$, by \cite[Theorem 4.4]{Vis}, the isomorphic class of first order deformations of $C$ is isomorphic to $\Ext^1(\Om_C,\cO_C)$. Hence the flat family of curves $\wt{\pi}:\cC^\ci\to \cP^\ci$ induces a morphism $T_{[S,C]}\cP^\circ \to \Ext^1(\Om_C,\cO_C)$.\\
\indent We can now establish a key result about the relation between the normal bundle of a K3 surface and the deformation theory of its hyperplane sections. This result is encoded in the following theorem.
\begin{theorem}
\label{main thm 1}
    Suppose $(S,\mathcal{L}) \in \mathcal{F}_{11}^\circ$, then $H^0(S,\mathcal{N}_{S/\mathbb{P}^{11}}(-1)) \geq 12$. Furthermore, for any curve $C \in |\mathcal{L}|$, we have $H^0(S,\mathcal{N}_{S/\mathbb{P}^{11}}(-1)) = 12$ if and only if the following two conditions hold:
\begin{enumerate}
    \item $\Hom(\Omega_C,\mathcal{O}_C)\allowbreak =0$.
    \item The morphism $T_{[S,C]}\cP^\circ \to \Ext^1(\Om_C,\cO_C)$ is injective.
\end{enumerate}
\end{theorem}
\begin{proof}
    Take an embedding $j:S \hookrightarrow \mathbb{P}^{11}$ such that $\mathcal{L}\cong j^*\mathcal{O}_{\mathbb{P}^{11}}(1)$. We have an exact sequence
    \begin{center}
        \begin{tikzcd}
            0 \ar[r] & H^0(S,j^*T_{\mathbb{P}^{11}}) \ar[r,"\vp"] & H^0(S,\mathcal{N}_{S/\mathbb{P}^{11}}) \ar[r] & H^1(S,T_S).
        \end{tikzcd}
    \end{center}
    \indent Consider the commutative diagram
    \begin{center}
    \begin{tikzcd}
        0 \arrow[r] & H^0(\PP^{11},\mathcal{O}_{\mathbb{P}^{11}}) \arrow[r]\arrow[d]  & H^0(\PP^{11},\mathcal{O}_{\mathbb{P}^{11}}(1)^{\oplus12}) \arrow[r]\arrow[d] & H^0(\PP^{11},T_{\mathbb{P}^{11}})\arrow[r]\arrow[d,"\psi"] & 0 \\
        0 \arrow[r] & H^0(S,\mathcal{O}_{S}) \arrow[r]  & H^0(S,\mathcal{O}_{S}(1)^{\oplus12}) \arrow[r] & H^0(S,j^*T_{\mathbb{P}^{11}}) \arrow[r] & 0 
    \end{tikzcd}
    \end{center}
    where the 2 rows are exact. Since the left two vertical maps are isomorphisms, the map $\psi:H^0(\mathbb{P}^{11},T_{\mathbb{P}^{11}}) \to H^0(S, j^*T_{\mathbb{P}^{11}})$ is an isomorphism.\\
    \indent Choose a coordinate $[x_0:\cdot\cdot\cdot:x_{11}]$ of $\PP^{11}$ such that $C=V(x_0) \cap S$. The deformation of $S$ as a genus $11$ K3 surface is obtained through deformation of $S$ in $\mathbb{P}^{11}$. The deformation of $C=V(x_0)$ in a fixed K3 surface $S$ is also obtained through deformation of $S$ in $\PP^{11}$. Therefore, the morphism $H^0(S,\mathcal{N}_{S/\mathbb{P}^{11}}) \to T_{[S,C]}\cP^\circ$ is surjective.\\
    \indent Suppose that $v\in\ker (H^0(S,\mathcal{N}_{S/\mathbb{P}^{11}}) \to T_{[S,C]}\cP^\circ)$, then $v$ induces a trivial deformation of $S$. Hence $v$ is in the image of $\vp:H^0(S, j^*T_{\mathbb{P}^{11}})\to H^0(S,\mathcal{N}_{S/\mathbb{P}^{11}})$, and also the image of $\vp\circ\psi$. Suppose $v=\vp(\psi(w))$ for some $w\in H^0(\PP^{11},T_{\PP^{11}})$.\\
    \indent By \cite[Theorem 4.4]{Vis},  $H^0(\PP^{11},T_{\PP^{11}})$ classifies automorphisms of the trivial first order deformation of $\PP^{11}$. The morphism $\vp\circ\psi$ induces a map
    \begin{align*}
        & F:\{\text{automorphisms of the trivial first order deformation of }\PP^{11}\}\\
        & \to\{\text{first order deformations of S in }\PP^{11}\}.
    \end{align*}
    Suppose $w$ corresponds to the automorphism
    $$\rho:\PP^{11}\times \Spec \CC[t]/(t^2)\to \PP^{11}\times \Spec \CC[t]/(t^2)$$ 
    of the trivial first order deformation of $\PP^{11}$, then from the construction in \cite{Vis},~$F(\rho)$ corresponds to the first order deformation 
    $$\rho(S\times\Spec \CC[t]/(t^2))\subset \PP^{11}\times \Spec \CC[t]/(t^2).$$
    Therefore, $F(\rho)$ induces a trivial deformation of $C\subset S$ if and only if $\rho$ induces a trivial deformation of $V(x_0) \su \PP^{11}$, or equivalently, 
    $$w\in K:=\ker(H^0(\PP^{11},T_{\PP^{11}})\to H^0(\PP^{10},
    i_0^*T_{\PP^{11}})\to  H^0(\PP^{10},
    \cN_{\PP^{10}/\PP^{11}})).$$
    Here, we denote $\PP^{10}:=V(x_0)$ and denote $i_0:\PP^{10}\hookrightarrow\PP^{11}$ the inclusion. Hence we have an isomorphism
    $H^0(S,\mathcal{N}_{S/\mathbb{P}^{11}})/\vp(\psi(K)) \cong T_{[S,C]}\cP^\ci$.\\
    \indent We have a natural morphism $h:K\to H^0(\PP^{10},T_{\PP^{10}})$. By a simple computation, we can show that $h$ is surjective and $\dim \ker h=12$.\\
    \indent Denote $i:C\hookrightarrow S$ and $j':C\hookrightarrow \PP^{10}$ the inclusions, then we have $i^*\mathcal{N}_{S/\mathbb{P}^{11}}^\vee \cong \mathcal{N}_{C/\mathbb{P}^{10}}^\vee$. Hence  $\mathcal{N}_{C/\mathbb{P}^{10}}^\vee$ is locally free. Since $C$ is generically smooth, we have a short exact sequence 
    $$0 \to \mathcal{N}_{C/\mathbb{P}^{10}}^\vee \to j'^*\Omega_{\mathbb{P}^{10}} \to \Omega_C \to 0.$$
    Therefore, we have a diagram
    \begin{center}
    \begin{tikzcd}
         & K \arrow[r,hook,"\vp\circ\psi"]\arrow[d,"g"] & H^0(S,\mathcal{N}_{S/\mathbb{P}^{11}})\arrow[d,"i^*"] & \\
        \Hom(\Omega_C,\mathcal{O}_C)\arrow[r,hook] & H^0(C,j'^*T_{\mathbb{P}^{10}}) \arrow[r,"f"] & H^0(C,\mathcal{N}_{C/\mathbb{P}^{10}}) \arrow[r,"\delta"] & \Ext^1(\Omega_C,\mathcal{O}_C)
    \end{tikzcd}
    \end{center}
    where $g$ is obtained by composing $h$ and $j'^*:H^0(\mathbb{P}^{10},T_{\mathbb{P}^{10}}) \to H^0(C,j'^*T_{\mathbb{P}^{10}})$. Suppose~$w\in K$ corresponds to the automorphism $\rho$ of the trivial first order deformation of~$\PP^{11}$, then we can check that $h(w)$ corresponds to $\rho|_{\PP^{10}\times \Spec \CC[t]/(t^2)}$ as an automorphism of the trivial first order deformation of $\PP^{10}$. Hence $f(g(w))$ corresponds to the first order deformation 
    $$\rho(C\times\Spec \CC[t]/(t^2))\su \PP^{10}\times \Spec\CC[t]/(t^2)$$
    of $C\su \PP^{10}$, which is also the first order deformation corresponding to $i^*(\vp(\psi(w)))$. Hence the above diagram commutes.\\
    \indent Using Lemma \ref{dimH^0(C,j*T_P10)=120}, we have that $g$ is surjective and $\ker g=\ker h$. Since $\vp\ci\psi$ is injective, we have 
    $$\dim H^0(S,\cN_{S/\cP^{11}}(-1))=\dim\ker i^*\geq\dim \ker g=12.$$
    Moreover, 
    \begin{align*}
       & \dim H^0(S,\cN_{S/\cP^{11}}(-1))=\dim\ker i^*=12\\
        \iff &\ker i^*=(\vp\ci \psi)(\ker g)\\
        \iff &\ker i^*=(\vp\ci \psi)(\ker f \circ g)\text{ and }\ker f \circ g=\ker g\\
        \iff &H^0(S,\mathcal{N}_{S/\mathbb{P}^{11}})/\vp(\psi(K))\to H^0(C,\cN_{C/\PP^{10}})/\im f\text{ is injective }
        \text{and }f\text{ is injective}\\
        \iff &T_{[S,C]}\cP^\circ \to \Ext^1(\Omega_C,\mathcal{O}_C)\text{ is injective and }\Hom(\Omega_C,\mathcal{O}_C)=0.\qedhere
    \end{align*}
\end{proof}
\indent Since the rational map $\phi$ given in diagram \ref{commutative diagram} is birational, there exists a dense open subset $U \subset \cP^\circ$ such that for each closed point $[S,C]\in U$, $C$ is a smooth curve, and $\phi|_U$ is an open embedding. For any closed point $[S,C] \in U$, using $T_{[C]}\cM_{11}\cong \Ext^1(\Om_C,\cO_C)$, we have that $\Hom(\Omega_C,\mathcal{O}_C)=0$ and $T_{[S,C]}\cP^\circ \to \Ext^1(\Om_C,\cO_C)$ is an isomorphism. Hence by Theorem \ref{main thm 1}, $H^0(S,\mathcal{N}_{S/\mathbb{P}^{11}}(-1)) = 12$.\\
\indent We now restrict $\mathcal{F}_{11}^\circ$ to the dense open subset $p(U)$, and by abuse of notation, continue to denote it as $\cF_{11}^\ci$. Consequently, we redefine $\cP^\circ:=p^{-1}(\mathcal{F}_{11}^\circ)$, $\cS_{11}^\circ:=\pi^{-1}(\cF_{11}^\circ)$, and~$\cC^\circ:=\wt{\pi}^{-1}(\cP^\circ)$. Then for any $[S,C]\in \cP^\circ$, by Theorem \ref{main thm 1}, $\Hom(\Omega_C,\mathcal{O}_C)=0$ and $T_{[S,C]}\cP^\circ \to \Ext^1(\Om_C,\cO_C)$ is injective. We redraw diagram \ref{commutative diagram} as follows:
\begin{equation}
\label{commutative diagram new}
\begin{tikzcd}
 & & \mathcal{C}_{11} \arrow[d]\\
 \cC^\circ \arrow[r,"\wt{\pi}"] \arrow[rru,dashed, bend left] \arrow[d,"\wt{p}"]& \mathcal{P}^\circ \arrow[r,dashed,"\phi"] \arrow[d,"p"]& \mathcal{M}_{11} \\
 \mathcal{S}_{11}^\circ \arrow[r,"\pi"]& \mathcal{F}_{11}^\circ. &
\end{tikzcd}
\end{equation}
\indent Next, we address the issue of curve singularities. To construct a well-behaved map from $\cP^\ci$ to the moduli space of curves, we need to control the singularities that appear in the family $\wt{\pi}:\cC^\ci\to \cP^\ci$. Recall that for an integral curve $C$, its delta invariant is defined as $\de(C):=p_a(C)-p_a(\wt{C})$, where $\wt{C}$ is the normalization of $C$ and $p_a$ is arithmetic genus. We begin with a characterization of curves with delta invariant one.
\begin{lemma}
\label{singularities with delta invariant one}
    For any integral curve $C$ over the field of complex numbers $\CC$ with only planar singularities, its delta invariant $\delta(C)=1$ if and only if $C$ has exactly one node or one cusp as its only singularity.
\end{lemma}
\begin{proof}
    The ``if'' direction is straightforward. We now prove the ``only if'' direction.\\
    \indent Obviously $C$ has exactly one singularity, denote it as $p$. Suppose $\cO_{C,p}\cong \Spec A$, and $m$ is the maximal ideal of $A$. Since $p$ is a planar curve singularity, the completion $\widehat{A}\cong \CC[[x,y]]/(f)$ for some formal polynomial $f\in \CC[[x,y]]$. Denote $m'$ the maximal ideal of $\CC[[x,y]]$. Then $f\in m'^2$.\\
    \indent By \cite[Theorem 10.5]{Mil}, we have $\mu=2\delta(C)-r+1$, where $r$ is the number of analytic branches passing through $p$, and $\mu$ is the intersection number of $f_x,f_y\in \CC[[x,y]]$, or equivalently, the length of $\CC[[x,y]]/(f_x,f_y)$ as a $\CC[[x,y]]$-module. It follows that $\mu\le2$.\\
    \indent If $f\in m'^3$, then $f_x,f_y\in m'^2$. So $\operatorname{length}\CC[[x,y]]/(f_x,f_y)\geq 3$, contradiction. Therefore, $f\notin m'^3$.\\
    \indent Suppose $f-q(x,y)\in m'^3$ where $q(x,y)$ is a quadratic form, then $q\neq0$. If $q$ is nondegenerate, then $p$ is a node by \cite[tag 0C4D]{Stack}. Otherwise, we can change the coordinates to assume $q=x^2$. Therefore, we can write
    $$f=x^2(1+p(x,y))+xy^2q(y)+y^3(d+r(y))$$
    where $p(x,y)\in \CC[[x,y]]$, $q(y),r(y)\in\CC[[y]]$, $d\in \CC$, and $p(0,0)=0$, $r(0)=0$. If $d=0$, we can check $\operatorname{length}\CC[[x,y]]/(f_x,f_y)\geq 3$, contradiction. Hence $d\neq 0$.\\
    \indent Now we change the coordinates
    $$z=(1+p(x,y))^\frac{1}{2}x+\frac{y^2q(y)}{2(1+p(x,y))^\frac{1}{2}},$$
    $$w=y(d+r(y)-\frac{yq(y)^2}{4(1+p(x,y))})^\frac{1}{3},$$
   then $\CC[[x,y]]/f\cong\CC[[z,w]]/\allowbreak(z^2+w^3)$ as topological local rings. Hence $p$ is a cusp.
\end{proof}
\indent We now show that curves with high delta invariants form a high-codimension subset in the parameter space $\cP^\ci$.
\begin{proposition}
\label{codimension of curve singularity geq its delta invariant}
    For any integer $\de_0>0$, there exists a dense open subset $V\su \cP^\circ$ such that
    \begin{enumerate}
        \item For any closed point $[S,C]\in V$, $C$ is an integral curve with $\de(C)\le \de_0$.
        \item $\codim \cP^\circ\backslash V>\de_0$.
    \end{enumerate}
\end{proposition}
\begin{proof}
Denote
$$A:=\big\{ p\in \cP^\circ \mid p=[S,C] \text{ is a closed point, and }
C \text{ is an integral curve with } \delta(C)>\delta_0 \big\}$$
    and denote $Z$ the Zariski closure of $A$ in $\cP^\ci$. Take $V:=\cP^\circ-Z$, then it suffices to prove $\codim Z>\de_0$.\\
    \indent Suppose $Z_1,...,Z_m$ are the irreducible components of $Z$, and denote $Z_i'$ the smooth locus of $Z_i$. It suffices to prove $\codim Z_i'>\de_0$ for each $1\le i\le m$.\\
    \indent Denote $\cC_i:=\wt{\pi}^{-1}(Z_i')$, and $\nu:\wt{\cC}_i\to \cC_i$ its normalization. Then the singular locus of $\wt{\cC}_i$ has codimension at least $2$. In addition, $\nu$ is a birational morphism. Hence we can choose dense open subset $Z_i^\circ$ of $Z_i'$, such that $Z_i^\circ$ is smooth, $(\wt{\pi}\circ\nu)|_{Z_i^\circ}$ is a smooth morphism, and for any $p\in Z_i^\circ$, $\nu_p:\wt{C}_p\to C_p$ is birational, where $\wt{C}_p:=(\wt{\pi}\circ\nu)^{-1}(p)$ and $C_p:=\wt{\pi}^{-1}(p)$. Hence $\nu_p$ is normalization for any $p\in Z_i^\circ$.\\
    \indent We only need to prove $\codim Z_i^\circ>\de_0$. Since $A$ is Zariski dense in $Z$, we can choose a closed point $p\in Z_i^\circ$ such that $\de(C_p)>\de_0$.\\
    \indent For any tangent vector $\vp:\Spec\CC[t]/(t^2)\to Z_i^\circ$ at $p$, the pullbacks of $\cC_i,\wt{\cC}_i$ along $\vp$ give a first order deformation of $\nu_p$. By \cite[Proposition 3.1]{Ran}, the isomorphism classes of first order deformations of $\nu_p$ is isomorphic to the group
    $$\Ext ^1(\nu_p^*\Omega_{C_p}\to \Om_{\wt{C}_p}, \nu_p^*\cO_{C_p}\to \cO_{\wt{C}_p})$$
    defined in \cite[Section 2]{Ran}. By \cite[Equation (2.2)]{Ran}, we have long exact sequence
    \begin{align*}
         &\Hom (\Om_{\wt{C}_p},\cO_{\wt{C}_p})\oplus \Hom(\Om_{{C_p}},\cO_{{C_p}})\to \Hom_{\nu_p}(\Om_{C_p},\cO_{\wt{C}_p})\\
         \to &
          \Ext^1(\nu_p^*\Omega_{C_p}\to \Om_{\wt{C}_p}, \nu_p^*\cO_{C_p}\to \cO_{\wt{C}_p})\\
          \xrightarrow{(g,h)}& \Ext^1 (\Om_{\wt{C}_p},\cO_{\wt{C}_p})\oplus \Ext^1(\Om_{{C_p}},\cO_{{C_p}})\xrightarrow{p+q}  \Ext_{\nu_p}^1(\Om_{C_p},\cO_{\wt{C}_p}).
    \end{align*}
    Here, we define $\Hom_{\nu_p}(F,G):=\Hom(F,{\nu_p}_*G)$ for any sheaf $F\in \mathrm{QCoh}(C_p)$ and sheaf~$G\in \mathrm{QCoh}(\wt{C}_p)$. For each $F\in \mathrm{QCoh}(C_p)$, 
    define $\Ext_{\nu_p}(F,\cdot)$ the right derived functor of $\Hom_{\nu_p}(F,\cdot):\mathrm{QCoh}(\wt{C}_p)\to \mathbf{Ab}$. Since $\nu_p$ is affine, ${\nu_p}_*:\mathrm{QCoh}(\wt{C}_p)\to \mathrm{QCoh}({C_p})$ is exact. Hence $\Ext_{\nu_p}^i(F,G)\cong \Ext^i(F,{\nu_p}_*G)$ for any $i\ge 0$.\\
    \indent We have a morphism 
    $$f:T_pZ_i^\circ\to \Ext ^1(\nu_p^*\Omega_{C_p}\to \Om_{\wt{C}_p}, \nu_p^*\cO_{C_p}\to \cO_{\wt{C}_p})$$
    and from the construction in \cite{Ran}, $g\ci f$ and $h\ci f$ correspond to the first order deformations of $\wt{C}_p$ and $C_p$, respectively. By Theorem \ref{main thm 1}, $h\circ f$ is injective. \\
    \indent Therefore, 
    \begin{align*}
        \dim T_pZ_i^\circ= &\dim \im g\circ f+\dim \ker g\circ f\\
        \le &\dim \Ext^1 (\Om_{\wt{C}_p},\cO_{\wt{C}_p})+\dim \im (h\circ f)|_{\ker g\circ f}\\
        \le &\dim \Ext^1 (\Om_{\wt{C}_p},\cO_{\wt{C}_p})+\dim \ker (q:\Ext^1(\Om_{{C_p}},\cO_{{C_p}})\to \Ext_{\nu_p}^1(\Om_{C_p},\cO_{\wt{C}_p}))\\
        \le &\dim \Ext^1 (\Om_{\wt{C}_p},\cO_{\wt{C}_p})+\dim \Hom (\Om_{C_p},\nu_*\cO_{\wt{C}_p}/\cO_{C_p}).
    \end{align*}
    For each singular point $x\in C_p$, $\Omega_{C_p,x}$ is generated by 2 elements as an $\cO_{C_p,x}$-module. Therefore,
    $$\dim \Hom_{\cO_{C_p,x}} (\Om_{C_p,x},(\nu_*\cO_{\wt{C}_p}/\cO_{C_p})_x)\leq 2\operatorname{length}((\nu_*\cO_{\wt{C}_p}/\cO_{C_p})_x).$$
    Summing over all singular points $x\in C_p$, we get
    $$\dim \Hom (\Om_{C_p},\nu_*\cO_{\wt{C}_p}/\cO_{C_p})\leq 2\de(C_p).$$
    Hence,
    $$\dim T_pZ_i^\circ\le (3g(\wt{C}_p)-3)+2\delta(C_p)=30-\de(C_p)<30-\de_0.$$
    It follows that $\codim Z_i^\ci>\de_0$, which proves the proposition.
\end{proof}
Combining Lemma \ref{singularities with delta invariant one} and Proposition \ref{codimension of curve singularity geq its delta invariant}, we get the following corollary:
\begin{corollary}
    \label{singularities worse than node and cusp has codim at least 2}
    There exists a dense open subset $V\su \cP^\circ$ such that $\codim \cP^\circ\backslash V\ge2$, and for each closed point $[S,C]\in V$, $C$ is an integral smooth or nodal or cuspidal curve with at most one singularity. 
\end{corollary}
\indent Since the curves parametrized by $V$ have only nodes and cusps as potential singularities, they naturally fit into the framework of pseudo-stable curves. We briefly recall the relevant definitions and properties of the moduli space of pseudo-stable curves.
\begin{definition}[\cite{Sch}]
     A connected curve $C$ of arithmetic genus $g \ge 2$ is called pseudo-stable if 
    \begin{enumerate}
        \item $C$ has only nodes and cusps as possible singularities.
        \item $C$ does not have elliptic tail. Here, an elliptic tail is an irreducible component $Z$ such that $p_a(Z)=1$ and $|Z \cap (C-Z)|=1$.
        \item $\omega_X$ is ample, or equivalently, for any irreducible component $Z$ with $p_a(Z)=0$, we have $|Z \cap (C-Z)|\ge 3$.
    \end{enumerate}
\end{definition}
\begin{theorem}[{\cite[Theorem 2.4]{BFMV} and \cite[Section 2]{HH}}]
\label{properties of moduli of p-stable curve}
    Assume $g \geq 3$. Define $\ov{\mathcal{M}}_{g}^p$ as the moduli stack of pseudo-stable curves of genus $g$. Then $\ov{\cM}_g^p$ is a smooth proper irreducible Deligne–Mumford stack of dimension $3g-3$, and its coarse moduli space $\ov{M}_g^p$ is an irreducible projective variety of dimension $3g-3$.
\end{theorem}
\begin{proposition}
\label{coarse moduli of pesudostable curves is normal}
    Assume $g\ge3$. The coarse moduli space $\ov{M}_g^p$ is normal.
\end{proposition}
\begin{proof}
    Suppose $\ov{M}_g^p$ is covered by affine open subsets $\{\Spec A_i\}$. Then by \cite[Theorem 11.3.1]{Ols}, $\Spec A_i\times_{\ov{M}_g^p}\ov{\cM}_g^p=[\Spec B_i/G_i]$ for a ring $B_i$ and a reduced finite group $G_i$, such that $B_i$ is a finite $A_i$-module. Hence $A_i\cong B_i^{G_i}$.\\
    \indent Since $\ov{\mathcal{M}}_{g}^p$ is smooth, each ring $B_i$ is integral and integrally closed. Hence $A_i\cong  B_i^{G_i}$ is integrally closed. Therefore, $\ov{M}_g^p$ is normal.
\end{proof}
By Corollary \ref{singularities worse than node and cusp has codim at least 2}, the family $\wt{\pi}^{-1}(V)\to V$ induces a morphism $\phi_V:V\to \ov{\cM}_{11}^p$. We now analyze the fibers of $\phi_V$, showing that they are zero-dimensional.
\begin{lemma}
\label{fiber of phi has dim 0}
    For any closed point $x \in |\ov{\cM}_{11}^p|$, its residual gerbe $\mathcal{G}_x$ is a closed substack, and $\dim \phi_V^{-1}(\cG_x)=0$.
\end{lemma}
\begin{proof}
     By Theorem \ref{properties of moduli of p-stable curve}, $\ov{\cM}_{11}^p$ is a Deligne–Mumford stack of finite type. Hence by \cite[tag 0H27]{Stack}, $\cG_x$ is a closed substack. From the construction of residual gerbe \cite[tag 06UI]{Stack}, we can see $\dim \cG_x=0$.\\
    \indent We have the Cartesian diagram
    \begin{center}
    \begin{tikzcd}
        \phi_V^{-1}(\mathcal{G}_x) \arrow[r,hook,"i'"]\arrow[d,"\phi_x"]\ar[dr, phantom, "\square"] & V \arrow[d,"\phi_V"]  & \\
        \mathcal{G}_x \arrow[r,hook,"i"] & \ov{\cM}_{11}^p.
    \end{tikzcd}
    \end{center}
    If $\dim \phi_V^{-1}(\mathcal{G}_x)>0$, then there exists a morphism $f:\Spec \mathbb{C}[t]/(t^2) \to \phi_V^{-1}(\mathcal{G}_x)$ which is nontrivial. Since $i'$ is a closed embedding, $i' \circ f$ is nontrivial. By Theorem \ref{main thm 1}, $\phi_V \circ i' \circ f$ is nontrivial. Hence $\phi_x \circ f$ is nontrivial.\\
    \indent Take an étale presentation $U \to \mathcal{G}_x$, then $U$ is reduced, finite type of dimension $0$. Hence $U$ is a union of finitely many reduced closed points. By the infinitesimal lifting criterion of étale morphism, there exists a morphism $g: \Spec \CC[t]/(t^2) \to U$ lifting~$\phi_x \circ f$. However, $g$ must be trivial, so is $\phi_x \circ f$, contradiction.
\end{proof}
\indent Using Lemma \ref{fiber of phi has dim 0}, we show that $V$ can be realized as an open subset of the coarse moduli space $\ov{M}_{11}^p$.
\begin{proposition}
\label{embed projective bundle into modili of pesudo-stable curve}
    Denote $\phi_V': V \to \ov{M}_{11}^p$ the composition of $\phi_V$ and the canonical morphism $\ov{\mathcal{M}}_{11}^p \to \ov{M}_{11}^p$, then $\phi_V'$ is an open embedding.
\end{proposition}
\begin{proof}
    By Theorem \ref{properties of moduli of p-stable curve}, $\ov{M}_{11}^p$ is a variety. So $\phi_V'$ is of finite type and separated.\\
    \indent For any closed point $x \in \ov{{M}}_{11}^p$, by \cite[tag 050B]{Stack}, $\phi_V|_{\phi_V'^{-1}(x)^{\mathrm{red}}}: \phi_V'^{-1}(x)^{\mathrm{red}} \to \ov{\mathcal{M}}_{11}^p$ factors through $\mathcal{G}_x$. So $\phi_V'^{-1}(x)^{\mathrm{red}}$ is a closed subscheme of $\phi_V^{-1}(\mathcal{G}_x)$. By Lemma \ref{fiber of phi has dim 0}, we have $\dim\phi_V'^{-1}(x)=0$.\\
    \indent Therefore, each fiber of $\phi_V'$ has dimension $0$. Since $\phi_V'$ is of finite type, $\phi_V'$ is quasi-finite.\\
    \indent By Zariski's Main Theorem, there exists a scheme $\ov{V}$ and commuting morphisms
    \begin{center}
    \begin{tikzcd}
        V \arrow[rr,hook,"j"]\arrow[dr,"\phi_V'"] && \ov{V} \arrow[dl,"g"]\\
        & \ov{M}_{11}^p &
    \end{tikzcd}
    \end{center}
    such that $j$ is an open embedding and $g$ is a finite morphism. Since $V$ is integral, we can replace $\ov{V}$ by the closure of $j(V)$ and assume $\ov{V}$ is integral.\\
    \indent Since $\phi_V'$ is birational, $g$ is birational. Using the fact that $g$ is finite, $\ov{V}$ is integral, and $\ov{M}_{11}^p$ is normal (Proposition \ref{coarse moduli of pesudostable curves is normal}), we get $g$ is an isomorphism. Hence $\phi_V'$ is an open embedding.
\end{proof}
Now we describe a birational morphism between the moduli of stable curves and the moduli of pseudo-stable curves, which is constructed in \cite{HH}. Recall that the divisor $\de_1$ in the moduli stack $\ov\cM_g$ is the locus of curves with elliptic tails, and the divisor $\De_1$ is the image of $\de_1$ under the natural morphism $\ov{\cM}_{g}\to \ov{M}_g$.
\begin{theorem}[{\cite[Theorem 1.1]{HH}}]
\label{divisorial contraction between moduli spaces}
     Assume $g\ge4$. Then there exists a morphism $\mathcal{T}: \ov{\cM}_{g}\to \ov{\cM}_{g}^p$. The morphism $\cT$ is an isomorphism on the complement of $\de_1$, and for each curve $C\in\de_1$, $\cT(C)$ is obtained by replacing each elliptic tail of $C$ with a cusp. Moreover, the induced morphism $\Upsilon: \ov {{M}}_{g} \to \ov {{M}}_{g}^p$ between coarse moduli spaces is a divisorial contraction. 
\end{theorem}
Taking $g=11$, we see that $\Upsilon(\De_1)$ is a closed subset of codimension at least $2$. Take $V':=V-\phi_V'^{-1}(\Upsilon(\De_1))$, then by Proposition \ref{embed projective bundle into modili of pesudo-stable curve}, $\codim \cP^\circ\ba V'\ge2$, and $\wt{\pi}^{-1}(V')\to V'$ is a family of nodal curves. Hence we get a morphism $\psi: V'\to \ov{\cM}_{11}$ and the induced morphism $\psi':V'\to \ov{M}_{11}$, such that $\cT\circ\psi=\phi_V|_{V'}$, $\Upsilon\circ\psi'=\phi_V'|_{V'}$. Since $\wt{\pi}^{-1}(V')\to V'$ is a family of integral curves, $\im \psi'\cap \De_1=\emptyset$. So $\psi'$ is an open embedding.\\
\indent To relate $\wt{\pi}^{-1}(V')$ with $\ov{M}_{11,1}$, we need to shrink $V'$ to restrict ourselves to curves with trivial automorphism groups. The following lemma shows that we can do this while removing a high-codimension subset.
\begin{lemma}
\label{singular locus of V has codim at least 2}
    There exists a dense open subset $W\su V'$ such that $\codim V'\ba W\ge9$, and for each closed point $p\in W$, the curve $\wt{\pi}^{-1}(p)$ has trivial automorphism group. 
\end{lemma}
\begin{proof}
    By \cite{Cor} (and its erratum \cite{Cor08}), for any $g\ge 3$, there exists a closed subset $S_g\subset M_g$ with codimension $g-2$, such that each closed point $x\in M_g\backslash S_g$ represents a curve with trivial automorphism group.\\
    \indent Consider the morphism $\xi:\ov{M}_{10,2}\to \ov{M}_{11}$ obtained by gluing the two marked points into a node. We have that $\xi$ surjects onto the canonical divisor $\De_0$, which is the locus of nodal curves, and $\xi(M_{10,2})$ is the dense open subset of $\De_0$ containing all integral curves with exactly one node. Consider the forgetful morphism $\pr:M_{10,2}\to M_{10}$. Then for any closed point $[C]\in \De_0\ba\xi(\mathrm{pr}^{-1}(S_{10}))$, its normalization $\wt{C}$ has trivial automorphism group. Hence $\Aut(C)$ is trivial.\\
    \indent By Corollary \ref{singularities worse than node and cusp has codim at least 2}, $\im \psi'\su M_{11}\cup \xi(M_{10,2})$. Take
    $$W:= V'\backslash \ov{\psi'^{-1}({S_{11}}\cup {\xi(\pr^{-1}(S_{10}))}}.$$
    Then for each closed point $p\in W$, the curve $\wt{\pi}^{-1}(C)$ has trivial automorphism group.\\
    \indent We have $\codim S_{11}=9$, and $\codim \xi(\pr^{-1}(S_{10}))=1+\codim_{M_{10}}S_{10}=9$. Hence $\codim V'\ba W\ge9$.
\end{proof}
Denote $\psi_W:W\to \ov{M}_{11}$ the restriction of $\psi'$ on $W$. Using Lemma \ref{singular locus of V has codim at least 2}, the pullback of the forgetful morphism $\pi_{11}:\ov{M}_{11,1}\to \ov{M}_{11}$ along $\psi_W$ is isomorphic to the family~$\wt{\pi}^{-1}(W)\to W$.\\
\indent In conclusion, we summarize the preceding constructions in Section \ref{Section: Reduction to moduli of curve} in the following theorem:
\begin{theorem}
\label{P11 is open subset of coarse noduli M11 after deleting codim 2 subsets}
    There exists a dense open subset $W\su \cP^\ci$ such that
    \begin{enumerate}
        \item The family $\wt{\pi}^{-1}(W)\to W$ is a flat family of nodal curves, which induces a commutative diagram
\begin{center}
    \begin{tikzcd}
    \wt{\pi}^{-1}(W)\arrow[r,"\wt{\psi}_W"]\arrow[d,"\wt{\pi}"] & \ov{M}_{11,1}\arrow[d,"\pi_{11}"]\\
    W\arrow[r,"\psi_W"]& \ov{M}_{11}.
\end{tikzcd}
\end{center}
        \item The diagram is Cartesian, and $\psi_W$, $\wt{\psi}_W$ are open embeddings.
        \item For any closed point $x\in W$, $\pi^{-1}(x)$ is an integral smooth or nodal curve with at most one node.
        \item $\codim \cP^\circ\backslash W\ge2$.
    \end{enumerate}
\end{theorem}

\section{Rationally connectedness and the tautological generation of $\CH^2(\ov{M}_{g,n})_\QQ$}
\label{Section: result on chow gp}
\indent The goal of this section is to prove Theorem \ref{CH2 is tautologically generated for rationally connected moduli space}, which is crucial for our proof of the generalized Franchetta conjecture in genus $11$ and is also of independent interest. The content of this section is self-contained.
\begin{theorem}
\label{CH2 is tautologically generated for rationally connected moduli space}
    Assume $2g-2+n>0$, and the coarse moduli space $\ov{M}_{g,n}$ is rationally connected. Then the second Chow group with rational coefficient $\CH^2(\ov{M}_{g,n})_\mathbb{Q}$ is generated by tautological classes.
\end{theorem}
\textbf{Convention.} Throughout this section, we work under the assumptions mentioned in Theorem \ref{CH2 is tautologically generated for rationally connected moduli space}. \\
\\
\indent To prove Theorem \ref{CH2 is tautologically generated for rationally connected moduli space}, we begin with the decomposition of the diagonal of $\overline{M}_{g,n}$.
\begin{proposition}
\label{decomposition of diagonal of coarse moduli}
     There exist closed subschemes $\Gamma_1,\Gamma_2\subset \ov{M}_{g,n} \times \ov{M}_{g,n}$ such that
     \begin{enumerate}
         \item $\Gamma_1=\{x\} \times \ov{M}_{g,n}$ for some closed point $x \in \ov{M}_{g,n}$, and $\Gamma_2\subset \ov{M}_{g,n} \times D$ for some codimension $1$ closed subscheme $D \subset \ov{M}_{g,n}$.
         \item There exists an integer $N$ such that 
     $$N[\Delta]=[\Gamma_1]+[\Gamma_2] \in \CH^{3g-3+n}(\ov{M}_{g,n} \times \ov{M}_{g,n})$$
     where $\Delta$ is the diagonal.
     \end{enumerate}
\end{proposition}
\begin{proof}
    For any field extension $\mathbb{C} \subset \Omega$ such that $\Om$ is algebraically closed, the base change~$(\ov{M}_{g,n})_\Omega$ is rationally connected. By \cite[Proposition 4.3.3(2)]{Kol} and \cite[Proposition 4.3.13(1)]{Kol}, we have $\CH_0((\ov{M}_{g,n})_\Omega)=\mathbb{Z}$.\\
    \indent The proposition then follows from \cite[Proposition 1]{BS}. Although $\ov{M}_{g,n}$ is not smooth, the proof of \cite[Proposition 1]{BS} still works for $\ov{M}_{g,n}$.
\end{proof}
By \cite[Corollary 2.10]{BP}, there exists a finite Galois cover $p:X \to \ov{M}_{g,n}$ such that $X$ is a smooth projective variety. Denote $G:=\Aut (X/\ov{M}_{g,n})$.\\
\indent The morphism $q:X \times X \to \ov{M}_{g,n} \times \ov{M}_{g,n}$ induced by $p$ is a flat morphism. Hence we get $q^*[\De]=q^*[\Gamma_1]+q^*[\Gamma_2] \in \CH^{3g-3+n}(X \times X)$. For each element $g\in G$, we denote $\De_X^g\subset X\times X$ the graph of the automorphism $g:X\to X$. Take a dense open subset~$U\subset \ov{M}_{g,n}$ such that $G$ acts freely on $p^{-1}(U)$. Then we have 
$$ q^{-1}(\De)\cap(p^{-1}(U)\times p^{-1}(U))=(\bigcup_{g\in G} \De_X^g)\cap (p^{-1}(U)\times p^{-1}(U)).$$
Since the generic point of each irreducible component of $q^{-1}(\De)$ is in $p^{-1}(U)\times p^{-1}(U)$, we have 
$[q^{-1}(\De)]=\sum_{g\in G}[\De_X^g]\in \CH^{3g-3+n}(X\times X)$.\\
\indent Therefore, we have the following corollary:
\begin{corollary}
\label{G-equivariant decomposition of diagonal}
    There exists an integer $N$ such that we have a decomposition $$N\sum_{g \in G} [\Delta_X^g]=[\Gamma_1']+[\Gamma_2'] \in \CH^{3g-3+n}(X \times X)$$
    where $\Gamma_1'=V'\times X$ for some closed subscheme $V'$ of dimension $0$, and $\Gamma_2' \subset X \times D'$ for some codimension $1$ closed subscheme $D'\subset X$.
\end{corollary}
Now we denote $\Phi:\CH^2(X)_{\alg} \to J^3(X)$ the restriction of the Abel-Jacobi map $\CH^2(X)_{\hom} \to J^3(X)$ on $\CH^2(X)_{\alg}$. We will show $\CH^2(X)_{\alg}^G\subset \CH^2(X)_{\tor}$ using the map $\Phi$.
\begin{lemma}
\label{ker(Phi) is torsion}
    $\ker(\Phi) \cap \CH^2(X)_{\alg}^G\subset \CH^2(X)_{\tor}$.
\end{lemma}
\begin{proof}
    By Corollary \ref{G-equivariant decomposition of diagonal}, $N\sum_{g\in G} [\Delta_X^g]_*=[\Gamma_1']_*+[\Gamma_2']_*$ as actions on $\CH^2(X)$. By \cite[Theorem A]{Mur} and \cite[Theorem C]{Mur} (which has an error and was corrected by \cite{Kah}), $\Phi$ is the universal regular homomorphism of $\CH^2(X)_{\alg}$.\\
    \indent Using the same argument as in \cite[Theorem 1(i)]{BS}, the map 
    $$N\sum_{g\in G} [\Delta_X^g]_*:\CH^2(X)_{\alg}\to \CH^2(X)_{\alg}$$
    factors through $\Phi$. Since the restriction of $N\sum_{g\in G} [\Delta_X^g]_*$ on $\CH^2(X)_{\alg}^G$ is multiplication by $N|G|$, the proposition follows.
\end{proof}
\begin{proposition}
\label{CH^2_alg is torsion}
    $\CH^2(X)_{\alg}^G\subset \CH^2(X)_{\tor}$.
\end{proposition}
\begin{proof}
    The Abel-Jacobi map of $X$ is $G$-equivariant. 
    Therefore, $\Phi(\CH^2(X)_{\alg}^G)\subset J^3(X)^G$. By Lemma \ref{ker(Phi) is torsion}, it suffices to prove $J^3(X)^G \subset J^3(X)_{\tor}$.\\
    \indent Consider the natural map $f:H^3(X,\mathbb{R}) \to J^3(X)$. We have $f$ is $G$-equivariant, surjective, and $\ker f=H^3(X,\ZZ)/H^3(X,\ZZ)_{\tor}$. Hence for any $a \in f^{-1}(J^3(X)^G)$ and~$g \in G$, we have $g^*(a)-a\in H^3(X,\mathbb{Z})/H^3(X,\ZZ)_{\tor}$. Therefore, we get the map
    $$\varphi:f^{-1}(J^3(X)^G) \to \operatorname{Mor}_\mathrm{Set}(G,H^3(X,\mathbb{Z})/H^3(X,\ZZ)_{\tor}),$$
    $$a \mapsto (g \mapsto g^*(a)-a).$$
    \indent The natural morphisms $H^3(\ov{M}_{g,n},\CC)\to H_3(\ov{M}_{g,n},\CC)^\vee$ and $H^3(X,\CC)\to H_3(X,\CC)^\vee$ are isomorphisms by the universal coefficient theorem. From the construction of these isomorphisms, we can see that $H^3(X,\CC)\to H_3(X,\CC)^\vee$ is $G$-equivariant. So it induces an isomorphism $H^3(X,\CC)^G\to H_3\allowbreak(X,\CC)_G^\vee$. By \cite[Theorem III.2.4]{Bre}, the pushforward $p_*$ induces an isomorphism $H_3(X,\CC)_G\cong H_3(\ov{M}_{g,n},\CC)$. Hence
    $$H^3(X,\CC)^G\cong H_3(X,\CC)_G^\vee\cong H_3(\ov{M}_{g,n},\CC)^\vee\cong H^3(\ov{M}_{g,n},\CC).$$
    By \cite[Proposition 36]{Beh} and \cite[Theorem 2.1]{AC}, $H^3(\ov{M}_{g,n},\mathbb{C})\cong H^3(\ov{\mathcal{M}}_{g,n},\mathbb{C})=0$. Hence we have $H^3(X,\mathbb{C})^G=0$, and $H^3(X,\RR)^G=0$. \\
    \indent Therefore, $\varphi$ is injective. Hence $J^3(X)^G$ is a countable subgroup of $J^3(X)$. Since $G$-action on $J^3(X)$ is continuous, $J^3(X)^G$ is a closed subgroup. The proof is now completed by Lemma \ref{closed countable subgroup of torus}.
\end{proof}
\begin{lemma}
\label{closed countable subgroup of torus}
    Any closed countable subgroup $G$ of the torus group $T^n$ is finite.
\end{lemma}
\begin{proof}
    Fix an isomorphism $T^n \cong (\mathbb{R}/\mathbb{Z})^n$, and denote $p_i:T^n \to \mathbb{R}/\mathbb{Z}$ the projection to the $i$-th factor. Since $G$ is compact, so is $p_i(G)$.\\
    \indent For each $1\le i\le n$, since $p_i(G)$ is a countable subgroup of $\RR/\ZZ$, it must be a finite subgroup. Hence $G$ is finite.
\end{proof}
\indent Now we can finish the proof of Theorem \ref{CH2 is tautologically generated for rationally connected moduli space}.
\begin{proof}[Proof of Theorem \ref{CH2 is tautologically generated for rationally connected moduli space}]
    We first show that every algebraically trivial cycle in $\CH^2(\ov{M}_{g,n})$ is torsion. For any class $a \in \CH^2(\ov{M}_{g,n})_{\alg}$, we have $p^*a\in \CH^2(X)_{\alg}^G$. Hence by Proposition \ref{CH^2_alg is torsion}, $p^*a \in \CH^2(X)_{\tor}$. Since $|G|a=p_*p^*a$, the class $a$ is torsion.\\
    \indent Next, we show that any homologically trivial cycle in $\CH^2(\ov{M}_{g,n})$ is torsion. For any class $a \in \CH^2(\ov{M}_{g,n})_{\hom}$, we have $p^*a\in \CH^2(X)_{\hom}^G$. Using the notation from Corollary \ref{G-equivariant decomposition of diagonal}, and viewing $[\Gamma_1']_*$ and $[\Gamma_2']_*$ as correspondences, we have the relation $N|G|p^*a=[\Gamma_1']_*p^*a+[\Gamma_2']_*p^*a$. Since $\Ga_1'=V'\times X$, and $V'\cdot p^*a=0\in \CH^*(X)$, we have $[\Ga_1']_*p^*a=0$.\\
    \indent Take $j: \wt{D} \to D'$ to be a resolution of singularities. Denote $j':X\tm \wt{D}\to X\tm D'$ the morphism induced by $j$, and denote $i:D'\hookrightarrow X$ the closed embedding. We can take $\wt{\Gamma}_2'\subset X \times \wt{D}$ such that $j'_*[\wt{\Ga}_2']=[\Ga_2']\in \CH^*(X\times D')$. Considering $[\wt{\Gamma}_2']_*$ as a correspondence between~$\CH^2(X)$ and $\CH^1(\wt{D})$, and $[\Ga_2']_*$ as a correspondence acting on~$\CH^2(X)$, we have $(i\circ j)_*[\wt{\Gamma}_2']_*p^*a=[\Gamma_2']_*p^*a$ by the projection formula. Since algebraic equivalence is equivalent to homological equivalence (after modulo torsion) for divisors on a smooth projective variety, it follows that $[\wt{\Gamma}_2']_*p^*a\in \CH^1(\wt{D})_{\alg,\QQ}$. Hence we have $[\Ga'_2]_*p^*a\in \CH^2(X)_{\alg,\QQ}$, which implies that 
    $$N|G|^2a=N|G|p_*p^*a=p_*[\Gamma_2']_*p^*a \in \CH^2(\ov{M}_{g,n})_{\alg,\QQ}.$$
    \indent We have $\CH^2(\ov{M}_{g,n})_{\alg,\QQ}=0$ by the first part of this proof. Therefore, any homologically trivial cycle in $\CH^2(\ov{M}_{g,n})$ is torsion.\\
    \indent By \cite[Theorem 1.5(1)]{CLP}, $H^4(\ov{M}_{g,n},\mathbb{Q})$ is generated by tautological classes. Since homologically trivial cycles are torsion, it follows that $\CH^2(\ov{M}_{g,n})_\QQ$ is generated by tautological classes. Hence we have proved Theorem \ref{CH2 is tautologically generated for rationally connected moduli space}.
\end{proof}
\section{Proof of the generalized Franchetta conjecture in genus 11}
\label{Section: Last section}
\indent In this section, we finish our proof of Theorem \ref{My work}. By a standard “spreading out” argument (see \cite[Chapter 1]{Voi}), it is equivalent to proving Theorem \ref{My work} for each closed point $x\in \cF_{11}^\ci$. Here, the construction of $\cF_{11}^\ci$ is mentioned before diagram \ref{commutative diagram new}.\\
\indent For each closed point $x \in \mathcal{F}_{11}^\circ$, we pullback the bottom left part of diagram \ref{commutative diagram new} along the inclusion $x\hookrightarrow \cF_{11}^\circ$ to obtain the following commutative diagram
\begin{center}
 \begin{tikzcd}
& Z \arrow[dl,hook,"i_3" ] \arrow[rr,"\wt{\pi}_x"] \arrow[dd,"\wt{p}_x"pos=0.25] & & \mathbb{P}_{11} \arrow[dl,hook,"i_1"] \arrow[dd,] \\ 
\cC^\circ \arrow[rr, crossing over,"\wt{\pi}"pos=0.75] \arrow[dd,"\wt{p}"] & & \cP^\circ \\
& S_x \arrow[dl,hook,"i_2"] \arrow[rr,"\pi_x"pos=0.3] & & x \arrow[dl,hook] \\
\mathcal{S}_{11}^\circ \arrow[rr,"\pi"] & & \mathcal{F}_{11}^\circ. \arrow[from=uu, crossing over, "p"pos=0.3] 
\end{tikzcd}
\end{center}
The morphisms are denoted as above. Since $\cF_{11}'$ is smooth, $\cF_{11}^\ci$ is also smooth. Therefore, all varieties in this cube are smooth, and we can define  Gysin pullbacks and intersection products on their Chow groups.\\
\indent The morphism $\wt{p}_x:Z \to S_x$ is a $\mathbb{P}^{10}$-bundle. Define $H:=\wt{\pi}_x^*[\mathcal{O}_{\mathbb{P}_{11}}(1)] \in \CH^1(Z)$, then by the projection bundle formula,
$$\CH^2(Z) \cong \wt{p}_x^*\CH^2(S_x) \oplus H \cdot \wt{p}_x^*\CH^1(S_x) \oplus H^2 \cdot \mathbb{Z}.$$
\indent For any element $a \in \CH^2(\mathcal{S}_{11}^\circ)$, we want to prove $i_2^*a \in \ZZ \cdot o_{S_x}$, where $o_{S_x}$ is the Beauville--Voisin class defined in \cite{BV}. By \cite{Roj}, the Chow group $\CH^2(S_x)$ is torsion-free. Hence, we only need to prove $i_2^*a \in \QQ\cdot o_{S_x}$ in $\CH^2(S_x)_\mathbb{Q}$.\\
\indent Considering the class $b:=\wt{p}^*a\in \CH^2(\cC^\circ)_\mathbb{Q}$, we only need to prove the following proposition.
\begin{proposition}
\label{final thm}
    For any class $b \in \CH^2(\cC^\circ)_\mathbb{Q}$, we have
    $$i_3^*b \in \QQ\cdot \wt{p}_x^*o_{S_x} \oplus H \cdot \wt{p}_x^*\CH^1(S_x)_\mathbb{Q} \oplus \QQ \cdot H^2.$$
\end{proposition}
Before we prove Proposition \ref{final thm}, we state the following two lemmas.
\begin{lemma}
\label{lemma1}
    For any class $b' \in \CH^2(\cP^\circ)_\mathbb{Q}$, the pullback $i_3^* \wt{\pi}^*b'\in\QQ\cdot H^2$.
\end{lemma}
\begin{proof}
    We have $i_3^*\wt{\pi}^*b'=\wt{\pi}_x^*i_1^*b'$. Since $i_1^*b'\in \QQ\cdot[\mathcal{O}_{\mathbb{P}_{11}}(1)]^2$, we have $\wt{\pi}_x^*i_1^*b'\in \QQ\cdot H^2$.
\end{proof}
\begin{lemma}
\label{lemma2}
    For any two classes $b_1,b_2 \in \CH^1(\cC^\circ)_\mathbb{Q}$, we have 
    $$i_3^*(b_1 \cdot b_2)\in \QQ\cdot \wt{p}_x^*o_{S_x} \oplus H \cdot \wt{p}_x^*\CH^1(S_x)_\mathbb{Q} \oplus \QQ \cdot H^2.$$
\end{lemma}
\begin{proof}
    By the projection bundle formula, we have $$\CH^1(Z)_\mathbb{Q} \cong \wt{p}_x^*\CH^1(S_x)_\mathbb{Q} \oplus \QQ\cdot H.$$
    By \cite[Theorem 1(3)]{BV}, the intersection product of any two divisors on $S_b$ is some multiple of $o_{S_x}$. Hence we have proved this lemma.
\end{proof}
\begin{proof}[Proof of Proposition \ref{final thm}]
 Using the notation in Theorem \ref{P11 is open subset of coarse noduli M11 after deleting codim 2 subsets}, suppose the irreducible components of $\cP^\circ\ba W $ are $Z_1,...,Z_n$, and suppose 
$$\codim Z_1=...=\codim Z_m=2<\codim Z_{m+1}\leq...\leq \codim Z_n$$
for some integer $1\le m\le n$. Denote $j:W\hookrightarrow \cP^\ci$ and $j':\pi^{-1}(W)\hookrightarrow\cC^\ci$ the open embeddings. Then we have the following commutative diagram
\begin{center}
    \begin{tikzcd}
    \bigoplus_{i=1}^m \CH^0(Z_i)_\mathbb{Q} \arrow[r]\arrow[d,"\wt{\pi}^*"] & \CH^2(\cP^\circ)_\mathbb{Q} \arrow[r,"j^*"]\arrow[d,"\wt{\pi}^*"] &\CH^2(W)_\mathbb{Q} \arrow[r]\arrow[d,"\wt{\pi}^*"] & 0\\
    \bigoplus_{i=1}^m \CH^0(\wt{\pi}^{-1}(Z_i))_\mathbb{Q} \arrow[r] & \CH^2(\cC^\ci)_\mathbb{Q} \arrow[r,"j'^*"] & \CH^2(\wt{\pi}^{-1}(W))_\mathbb{Q} \arrow[r] & 0
\end{tikzcd}
\end{center}
where the two rows are exact.\\
\indent Since the morphism $\wt{\psi}_W$ in Theorem \ref{P11 is open subset of coarse noduli M11 after deleting codim 2 subsets} is an open embedding, we have
$$\wt{\psi}_W^*:\CH^*(\ov{M}_{11,1})_\QQ\to\allowbreak \CH^*(\pi^{-1}(W))_\QQ$$
is surjective. By \cite[Corollary 2]{Bar}, $\mathcal{M}_{11,1}$ is unirational, so is $\ov{M}_{11,1}$. Hence $\ov{M}_{11,1}$ is rationally connected. By Theorem \ref{CH2 is tautologically generated for rationally connected moduli space}, $\CH^2(\ov{M}_{11,1})_\mathbb{Q}$ is generated by tautological classes. Therefore, $\CH^2(\pi^{-1}(W))_\mathbb{Q}$ is generated by the pullback of tautological classes in $\CH^2(\ov{M}_{11,1})_\QQ$.\\
\indent If $a\in \CH^2(\ov{M}_{11,1})_\QQ$ is a tautological class that can be represented by a cycle supported on $\ov{M}_{11,1}\ba \wt{\psi}_W(\wt{\pi}^{-1}(W))$, then $\wt{\psi}_W^*a=0$. Using Theorem \ref{P11 is open subset of coarse noduli M11 after deleting codim 2 subsets}(3), we can see that $\CH^2(\pi^{-1}(W))_\mathbb{Q}$ is generated by the pullback of the following $7$ classes in $\CH^2(\ov{M}_{11,1})_\QQ$ along $\wt{\psi}_W$:
$$\kappa_1^2\text{, } \kappa_1\psi_1\text{, } \psi_1^2\text{, } \kappa_2\text{, }{\xi_{G_1}}_*[\ov{M}_{G_1},\kappa_1] \text{, }  {\xi_{G_1}}_*[\ov{M}_{G_1},\psi_1] \text{, }  {\xi_{G_1}}_*[\ov{M}_{G_1},\psi_2]=\xi_{G_1,*}[\ov{M}_{G_1},\psi_3].$$
Here, $G_1$ is the stable graph corresponds to integral pointed curves of genus $11$ with one node and one labeled point, and the morphism $\xi_{G_1}:\ov{M}_{G_1}\to \ov{M}_{11,1}$ is the canonical morphism defined in \cite[Appendix A]{GR}. We have an isomorphism $\ov{M}_{G_1}\cong \ov{M}_{10,3}$, and we choose the labeled points such that under the morphism $\xi_{G_1}$, the labeled points $2,3$ in~$\ov{M}_{10,3}$ correspond to the node. Hence we can define the tautological classes $[\ov{M}_{G_1},\kappa_1]$ and $[\ov{M}_{G_1},\psi_i]$ ($1\le i\le3$) on $\ov{M}_{G_1}$.\\
\indent For each $1\le i\le m$, the morphism $\pi^{-1}(Z_i)\to Z_i$ is flat with geometrically integral fibers. Therefore, $\pi^{-1}(Z_i)$ is irreducible. Hence $\CH^0(\pi^{-1}(Z_i))_\QQ$ is generated by the class $[\pi^{-1}(Z_i)]$.\\
\indent We now claim that for each class $b\in \CH^2(\cC^\ci)_\QQ$, if we can write $j'^*b$ in the form
$$j'^*b=\sum_{k=1}^lb_k^1\cdot b_k^2+\wt{\pi}^*b'$$ 
for some integer $l$, some classes $b_k^1,b_k^2\in \CH^1(\pi^{-1}(W))_\QQ$ ($1\le k\le l$) and $b'\in \CH^2(W)_\QQ$, then we have 
$$i_3^*b \in \QQ\cdot \wt{p}_x^*o_{S_x} \oplus H \cdot \wt{p}_x^*\CH^1(S_x)_\mathbb{Q} \oplus \QQ \cdot H^2.$$
Indeed, since $j'^*:\CH^1(\cC^\ci)_\QQ\to \CH^1(\pi^{-1}(W))_\QQ$ is an isomorphism, we can define the classes $c_k^i:={j'^*}^{-1}(b_k^i)$ for $1\le k\le l$ and $i=1,2$. We can choose $c'\in \CH^2(\cP^\ci)_\QQ$ such that $j^*c'=b'$. Then we have $j'^*(b-\sum_{k=1}^l c_k^1\cdot c_k^2-\wt{\pi}^*c')=0$. Hence 
$$b=\sum_{k=1}^l c_k^1\cdot c_k^2+\wt{\pi}^*c'+\sum_{i=1}^m a_i\wt{\pi}^*[Z_i]$$
for some rational numbers $a_1,...,a_m\in\QQ$. Using Lemma \ref{lemma1} and Lemma \ref{lemma2}, the claim is proved.\\
\indent It remains to show that the restriction of the $7$ tautological classes can be written in the form $$\sum_{k=1}^lb_k^1\cdot b_k^2+\wt{\pi}^*b'.$$
It is straightforward to see that the classes $\wt{\psi}_W^*\ka_1^2$, $\wt{\psi}_W^*\ka_1\psi_1$ and $\wt{\psi}_W^*\psi_1^2$ can be written in this form.\\
\indent Denote $\pi_{11}:\ov{M}_{11,1}\to \ov{M}_{11}$ the forgetful morphism, and denote $[\ov{M}_{11},\ka_2]$ the $\ka_2$-class on $\ov{M}_{11}$. Then we have the following formulas on $\CH^*(\ov{M}_{11,1})_\QQ$:
$$\ka_2=\pi_{11}^*[\ov{M}_{11},\ka_2]-\psi_1^2,$$
$${\xi_{G_1}}_*\ka_1=\ka_1\cdot {\xi_{G_1}}_*[\ov{M}_{G_1}],$$ 
$${\xi_{G_1}}_*\psi_1=\psi_1\cdot {\xi_{G_1}}_*[\ov{M}_{G_1}].$$
Using the notation and result in Theorem \ref{P11 is open subset of coarse noduli M11 after deleting codim 2 subsets}, we have $\wt{\psi}_W^*\pi_{11}^*=\wt{\pi}^*\psi_W^*$. Therefore, the classes $\wt{\psi}_W^*\ka_2$, $\wt{\psi}_W^*{\xi_{G_1}}_*\ka_1$ and $\wt{\psi}_W^*{\xi_{G_1}}_*\psi_1$ can be written in the form $\sum_{k=1}^lb_k^1\cdot b_k^2+\wt{\pi}^*b'$.\\
\indent Apply \cite[Equation 11]{GR} to the stable graph $G_1$, and pullback the two sides of the equation along $\wt{\psi}_W$. After canceling the cycles that are supported on $\ov{M}_{11,1}\ba \wt{\psi}_W(\wt{\pi}^{-1}(W))$, we get
$$\wt{\psi}_W^*{\xi_{G_1}}_*\psi_2=-\frac{1}{2}\wt{\psi}_W^*({\xi_{G_1}}_*[\ov{M}_{G_1}])^2.$$
Therefore, $\wt{\psi}_W^*{\xi_{G_1}}_*\psi_2$ can also be written in the form $\sum_{k=1}^lb_k^1\cdot b_k^2+\wt{\pi}^*b'$.\\
\indent In conclusion, we have proved this proposition. 
\end{proof}
Therefore, we complete the proof of Theorem \ref{My work}.

\end{document}